\documentclass{amsart}

\usepackage[utf8]{inputenc}

\usepackage{amsthm}
\usepackage{amsmath}
\usepackage{amsfonts}
\usepackage{ amssymb }

\usepackage{hyperref}
\usepackage{graphicx}
\usepackage{graphicx}

\newtheorem{theorem}{Theorem}[section]
\newtheorem*{theorem*}{Theorem}
{\theoremstyle{definition}
\newtheorem{defn}[theorem]{Definition}}
\newtheorem{lemma}[theorem]{Lemma}

\newtheorem{corollary}[theorem]{Corollary}
{\theoremstyle{remark}
\newtheorem*{remark}{Remark}}
{\theoremstyle{remark}
\newtheorem*{convention}{Convention}}

\newcommand{\h}{\mathbb{H}}
\newcommand{\R}{\mathbb{R}}
\newcommand{\C}{\mathbb{C}}
\newcommand{\Q}{\mathbb{Q}}
\newcommand{\Z}{\mathbb{Z}}

\makeatletter
\newtheorem*{rep@theorem}{\rep@title}           
\newcommand{\newreptheorem}[2]{%
\newenvironment{rep#1}[1]{%
 \def\rep@title{#2 \ref*{##1}}%
 \begin{rep@theorem}}%
 {\end{rep@theorem}}}
\makeatother

\newreptheorem{theorem}{Theorem}

\title{Systole Length in Hyperbolic $n$-Manifolds}
\author{Joe Scull}

\begin{document}
\thanks{This work was supported by The Engineering and Physical Sciences Research Council (EPSRC) under grant EP/R513295/1 studentship 2100094}

\begin{abstract}
We show that the length $R$ of a systole of a closed hyperbolic $n$-manifold $(n \geq 3)$ admitting a triangulation by $t$ $n$-simplices can be bounded below by a function of $n$ and $t$, namely 
\[ R \geq \frac{1}{2^{(nt)^{O(n^4t)} }} .\]
We do this by finding a relation between the number of $n$-simplices and the diameter of the manifold  and by giving explicit bounds for a well known relation between the length of the core curve of a Margulis tube and its radius.

We prove the same result for finite volume manifolds, with a similar but slightly more involved proof.

\end{abstract}

\maketitle

\section{Introduction}
\label{SecIntro}

A major tool used to understand Riemannian manifolds is examining how different geometric properties of such manifolds relate to one another or conversely identifying when there is no such relation.
Gromov \cite{Gromov78} proves a fundamental relation between volume and diameter for Riemannian manifolds $M$ of negative curvature and dimension $n \geq 8$, proving that if sectional curvature of $M$ lies in the interval $(-1,0)$, then 
\[ vol(M) \geq C_n(1 + diam(M))  \]
where $C_n$ is some constant dependent on $n$. A similar but slightly weaker result is also proven for $n\geq 4$. Conversely, such a result cannot hold in dimension $3$ as there exist sequences of hyperbolic 3-manifolds with volume uniformly bounded whose diameters tend to infinity, a fact which follows from Thurston's hyperbolic Dehn surgery theorem \cite{Thurston78}.
Influential to Gromov's paper was one almost a decade earlier by Cheeger \cite{Cheeger70} in which he finds a lower bound on injectivity radius as a function of volume, diameter and a lower bound on sectional curvature.
He goes on to use this relationship to prove the famous Cheeger finiteness theorem, a modern statement of which can be found in  \cite{Chavel95}.

For hyperbolic manifolds of dimension $\geq 3$, Margulis' Lemma and the resulting thick-thin decomposition  tell us that having a sufficiently short systole is indicative of the existence of Margulis tubes, whose diameters grow as the length of their core curves decrease.
In this paper we give an explicit lower bound on how the diameter of these tubes relate to the length of the core curve, and further show how this depends on $n$.
\newpage

\begin{reptheorem}{InjRadDiameter}
Suppose that $M$ is a finite volume hyperbolic $n$-manifold  $(n \geq 3)$ with systole(s) of length $R \leq 2\epsilon_n$, where $\epsilon_n$ is the Margulis constant in dimension $n$. Then the distance from a systole to the boundary of the Margulis tube containing it is bounded below by \[\frac{1}{n} log\left(\frac{1}{R}\right) +  log (\epsilon_n) -log( 4). \]
\end{reptheorem}

A less explicit version of this theorem is surely known to experts.  In dimensions $\geq 4$ Reznikov \cite{Reznikov95} shows that for a closed hyperbolic manifold a small injectivity radius implies  that the manifold has both a large diameter and volume, but how this depends on dimension isn't proven. A similar result for $n \geq 8$ follows from our Theorem \ref{InjRadDiameter} and Gromov's result above.
Belolipetsky and Thomson \cite{BelolipetskyThomson11} also prove a similar result for some infinite families of hyperbolic manifolds whose injectivity radii tend to 0. It is worth noting that their result, which builds on earlier work of Agol \cite{Agol06}, demonstrates the existence of hyperbolic manifolds with arbitrarily small systoles and thus proves that there can be no universal lower bound on systole length. In our second main theorem we show that triangulation complexity can be used to bound from below the length of a systole.

\begin{theorem}\label{Thm1.1}
Given a closed (or finite volume) hyperbolic $n$-manifold $(n \geq 3)$ triangulated by $t$ $n$-simplices (possibly semi-ideal), the length $R$ of a systole of $M$ is bounded below by a function of $t$ and $n$, in particular 
\[ R \geq \frac{1}{2^{(nt)^{O(n^4t)} }} \]
\end{theorem}

\begin{remark}
Note that for a non-orientable hyperbolic $n$-manifold, we can apply Theorem \ref{Thm1.1} to its orientable double cover and use that to bound the length of systoles in the original manifold. In doing so we double the number of tetrahedra and halve the bound on systole length. This difference disappears when we simplify big O notation however. 

Because of this, we shall from this point on assume all our manifolds to be orientable, and thus all our lattices shall lie in Isom$^+(\h^n)$.

\end{remark}
 
 It's worth noting that at first, one might think this theorem for $n \geq 4$ follows from Reznikov's result but his theorem produces a constant which varies with $n$ in an unknown way.

In dimension 3 Theorem \ref{Thm1.1} forms a crucial part of the  argument in an upcoming paper of the author \cite{Scull21} on a bounded runtime algorithm for the homeomorphism problem for hyperbolic 3-manifolds. That paper uses this same idea of finding a link between the combinatorial information of a triangulation and geometric information about the manifold and both papers are inspired by Kuperberg's work \cite{Kuperberg15} and his use of the results of Grigoriev and Vorobjov. 

In fact in dimension $3$ a lot more is known. Futer, Purcell and Schleimer \cite{FuterPurcellSchleimer19} prove a $3$-dimensional version of Theorem \ref{InjRadDiameter} which they prove to be sharp up to the additive constants and which is still linear in log$(1/R)$. Kalelkar and Raghunath \cite{KalelkarRaghunath20} also provide a result linking the number of tetrahedra in a triangulation of a cusped hyperbolic $3$-manifold $M$ to systole length. Their result gives much better bounds in terms of $t$ but requires that the original triangulation be geometric and also depends on the geometry of the triangulation, which is quite a strong restriction.  
  A closed 3-dimensional version of Theorem \ref{Thm1.1} follows from a combination of Theorem \ref{InjRadDiameter} (or indeed \cite{FuterPurcellSchleimer19}) and another result relating combinatorial and geometric properties of hyperbolic 3-manifolds, due to Matthew White \cite{White00a} (see also \cite{White00b} \cite{White02}).
\begin{theorem*}[Theorem 5.9, \cite{White00b}]
There is an explicit constant $K>0$ such that if $M$ is a closed, connected, hyperbolic 3-manifold, and 
\[ P= \langle x_1, \ldots, x_n \mid r_1, \ldots , r_m \rangle \]
 is a presentation of its fundamental group, then $diam(M) < K(l(P))$, where
 \[ l(P) =  \sum^m_{i=1} l(r_i) \]
 and $l(r_i)$ is the word length of a given relator.
\end{theorem*} 

The 3-dimensional version of Theorem \ref{Thm1.1} that follows from this combination has a tighter bound on systole length of the form

\[ R \geq \frac{1}{2^{O(t)} }.\]
The proof follows from considering a natural presentation of the fundamental group induced by the given triangulation. For further examples of low-dimensional results in this area, see Adams and Reid \cite{AdamsReid00}, Kalelkar and Phanse \cite{KalelkarPhanse19} or the aforementioned work of Agol \cite{Agol06}.

It's also worth noting that in dimension two such a result does not hold, in fact  a given surface of negative euler characteristic admits hyperbolic structures with arbitrarily small curves. This is shown in any introductory text on Teichm\"uller space, see for example \cite{FarbMargalit11} \cite{Martelli16}.


\subsection{Sketch of the Paper}
 
 In Section \ref{SecPolynomial} we recount the work of Grigoriev and Vorobjov which bounds the size of a solution to a system of polynomials as a function of the size of the system. In Section \ref{SecHomEq} we use the fact that cocycles into $Isom^+(\h^n)$ are solutions to some natural system of polynomials and the results of Section \ref{SecPolynomial} to find a cocycle which is ``bounded". This cocycle corresponds to some faithful lattice representation of the fundamental group of a given $n$-manifold which admits a hyperbolic structure. We can then use this cocycle to define a homotopy equivalence from a triangulation of this manifold to some specific hyperbolic $n$-manifold $X$. This homotopy equivalence has as its image a union of geometric simplices, and the bounds on the cocycle control the geometry of these simplices, in particular their diameters are bounded.
 
In Section \ref{SecInjRad} we use Margulis' Lemma to calculate a relationship between a systole's length and the diameter of the Margulis tube containing the systole. We then use the bounds on the geometry of the simplices to bound the diameter of a closed hyperbolic manifold $X$ and hence to bound the length of its systoles.

In Section \ref{SecCusp} we adapt some of the results used above to the finite volume case and apply the same result on Margulis tubes. To do this we find a bound on the diameter of a subset of the manifold which contains all the Margulis tubes.


\section{Systems of Polynomial Inequalities} \label{SecPolynomial}

What we present here is adapted from \cite{Grigoriev86} (which is a survey of the results in \cite{GrigorievVorobjov88}) with a few corollaries which show how to adapt the theorem to our situation. This section shows how given a system of polynomial inequalities, one can find solutions which are bounded in terms of the `size' of the system. \\ \medskip
Let a system of polynomial inequalities
\begin{equation*}
f_1 > 0, \ldots f_m > 0, f_{m+1} \geq 0, \ldots , f_{\kappa} \geq 0
\end{equation*}
be given, where the polynomials $f_i \in \Q[X_1, \ldots, X_N]$ satisfy the bounds
\begin{equation*}
    deg_{X_1,\ldots,X_N}(f_i) < d, \qquad l(f_i) < M, \qquad 1 \leq i \leq \kappa.
\end{equation*}
Here $deg_{X_1,\ldots,X_N}(f_i)$ is the maximum degree of a monomial in $f_i$, where the degree of a monomial is the sum of the exponents of the monomial. The length, or complexity $l$ is defined on rational numbers by $l (\frac{p}{q}) = log_2(\lvert pq \rvert +2)$ and on polynomials it is defined as the maximum complexity among its coefficients.

Let $\alpha = (\alpha_1, \ldots, \alpha_N)$ be a solution to the system of inequalities where each $\alpha_i$ is an algebraic number. Then then we can represent each $\alpha_i$ in terms of a primitive element, $\theta$, of the field $\Q (\alpha_1, \ldots, \alpha_N) = \Q(\theta)$. We represent $\theta$ by providing an irreducible polynomial $\Phi (X) \in \Q (X)$ of which $\theta$ is a root and an interval $(\beta_1, \beta_2) \subseteq \Q$ with endpoints in $\Q$ which determines $\theta$ among the roots of $\Phi$. With $\theta$ defined, one has $\alpha_i = \sum_j \alpha_i^{(j)} \theta^j$ for $\alpha_i^{(j)} \in \Q$. Using this notation we can now state the theorem of Grigoriev.

\begin{theorem}[Grigoriev \cite{Grigoriev86}] \label{Grigoriev}

For a given system of inequalities such as those above, for each connected component of the solution set, there exists a solution $\alpha = (\alpha_1, \ldots, \alpha_N) $ represented as above for which the following are true:
\begin{center}

$deg(\Phi) \leq (\kappa d )^{O(N)}; \qquad l(\Phi), l(\alpha_j^{(i)}), l(\beta_1), l(\beta_2) \leq M (\kappa d)^{O(N)}$
\end{center}
\end{theorem}

The following corollary follows the same method of proof as that of Lemma 8.9 in \cite{Kuperberg15}, but for self-containedness, we provide it here.
\begin{corollary}
$\frac{1}{2^{M (\kappa d)^{O(N)}}} \leq \lvert \theta \rvert \leq 2^{M (\kappa d)^{O(N)}}$
\end{corollary}

\begin{proof}
Note that we can scale $\Phi$ so that it is an integer polynomial without negating the statement that $l(\Phi)\leq M (\kappa d)^{O(N)}$. So we may assume that the coefficients, $\gamma_i$ of $\Phi$ are integers, and hence the $\gamma_i$ have size bounded by $2^{l(\Phi)} -2$
Now note that as $\Phi (\theta) = 0$ we get that 
\[ -\gamma_{deg(\Phi)} \theta^{deg(\Phi)} = \gamma_0 + \ldots +\gamma_{deg(\Phi)-1}\theta^{n-1} \]
and so 
\[ \lvert \theta \rvert = \left\lvert \frac{\gamma_0 + \ldots +\gamma_{deg(\Phi)-1}\theta^{n-1} }{\theta^{n-1}\gamma_{deg(\Phi)}} \right\rvert  \leq \sum_{i=0}^{deg(\Phi)-1} \lvert \gamma_i \rvert\ \] 
Note that $\theta^{-1}$ is a root of 
\[ \Psi(x) = x^n \gamma_0 + \ldots + \gamma_{deg(\Phi)} = x^n \Phi(x^{-1}) \]
and so 
\[ \frac{1}{\lvert \theta \rvert} = \lvert \theta^{-1} \rvert \leq \sum_{i=0}^{deg(\Phi)} \lvert \gamma_i \rvert .\]
Applying the bounds on the $\gamma_i$ gives
\begin{equation*}
\lvert \theta \rvert \leq  \sum_{i=0}^{deg(\Phi)-1} \gamma_i \leq \sum_{i=0}^{deg(\Phi)-1} 2^{l(\Phi)} \leq deg(\Phi) 2^{l(\Phi)}
\end{equation*}
and similarly,
\begin{equation*}
\lvert \theta \rvert \geq  \frac{1}{deg(\Phi) 2^{l(\Phi)}}
\end{equation*}
Finally, Theorem \ref{Grigoriev} tells us that 
\begin{equation*}
    deg(\Phi) 2^{l(\Phi)} \leq (\kappa d )^{O(N)} 2^{M (\kappa d)^{O(N)}} \leq 2^{M (\kappa d)^{O(N)}}
\end{equation*}
and the statement follows.
\end{proof}
From this, we can derive a statement about the size of the solutions themselves.

\begin{corollary} \label{variablebound}
Let $\alpha$ be as in Theorem \ref{Grigoriev}, then for each $i$, if $\alpha_i \neq 0$
\begin{equation*}
    \frac{1}{2^{M ((\kappa+2N) d)^{O(N)}}} \leq \lvert \alpha_i \rvert \leq 2^{M (\kappa d)^{O(N)}} .
\end{equation*}
\end{corollary}

\begin{proof}
First see that we can achieve an upper bound on $\lvert \alpha_i \rvert$ by applying the triangle inequality in the following manner:
\begin{equation*}
    \lvert \alpha_i \rvert 
    \leq \lvert \sum_{j=0}^{deg(\Phi)}\alpha_i^{(j)} \theta^j \rvert 
    \leq \sum_{j=0}^{deg(\Phi)}\lvert \alpha_i^{(j)} \rvert \lvert \theta^j \rvert 
    \leq \sum_{j=0}^{deg(\Phi)} 2^{M (\kappa d)^{O(N)}} 2^{M (\kappa d)^{O(N)}} 
\end{equation*}
\begin{equation*}
    \leq ((\kappa d)^{O(N)}+1)  2^{M (\kappa d)^{O(N)}} 2^{M (\kappa d)^{O(N)}}
    \leq  2^{M (\kappa d)^{O(N)}} .
\end{equation*}
Note that a lower bound on $\lvert \alpha_i\rvert$ is equivalent to an upper bound on $\lvert \alpha_i^{-1} \rvert$, so we can modify our system of polynomial inequalities by adding some $\beta_i$ such that $\alpha_i \beta_i = 1$ for all $i$. This is represented by two inequalities and one new variable for each of the original variables. Now applying the theorem to this new system of inequalities gives an upper bound for $\lvert \beta_i \rvert$ as above except $\kappa$ is replaced by $\kappa +2N$, as $O(2N)$ and $ O(N)$ are the same. Taking reciprocals of both sides gives us the lower bound in the statement.
\end{proof}

Before we move on we should note that the theorem requires that the system be given as polynomial inequalities. When we create our system, we shall allow for equalities as well, which can be replaced by a pair of inequalities. Because of big O notation, doubling the number of polynomials will not affect the final bound and so we shall ignore this distinction from here on.

\section{Cocycles and Homotopy Equivalences} \label{SecHomEq}

\begin{defn}[Triangulation]
A triangulation of an $n$-manifold $M$ is a simplicial complex equipped with a homeomorphism to $M$.
\end{defn}

\begin{remark}
The results of this paper also hold for a more general definition of triangulation often used by low dimensional topologists, at least in dimension $3$. In higher dimensions, one tends to need to introduce further requirements to avoid pathologies, and so we shall stick to the simplicial setting. For a more detailed discussion see Thurston's definition of a rectilinear gluing and further comments in  \cite{Thurston97}. 
In support of this remark efforts have been made throughout the paper to avoid the use of certain properties of simplicial complexes that aren't shared by more general triangulations, for example that two $n$-simplices intersect in at most one face.
\end{remark}

\begin{defn}[Star and Link]
The star st$(\Delta)$ of a simplex $\Delta$ in a simplicial complex $K$ is the smallest subcomplex of $K$ containing all simplices which have $\Delta$ as a face.
The link lk$(\Delta)$ of a simplex $\Delta$ in a simplicial complex $K$ consists of all simplices in st$(\Delta)$ which do not contain $\Delta$.
\end{defn}

\begin{defn}[Simplicial Path]
A sequence of oriented 1-simplices in a triangulation such that the end vertex of each simplex is the start vertex of the next is called a simplicial path.
\end{defn}

\begin{defn}
Let $M$ be a triangulated $n$-manifold  and $\Lambda$ a group.
A cocycle $\alpha \in C^1(M,\Lambda)$ is a map from the oriented $1$-simplices of $M$ to $\Lambda$ which satisfies the following conditions

\begin{itemize}
\item For each 2-simplex of $M$, let $a,b,c$ be the three edges of the simplex oriented such that $a$ and $c$ start at the same vertex, and $b$ and $c$ end at the same vertex, then $\alpha(a) \alpha(b) =  \alpha(c)$ 
\item if $\bar{a}$ is the same edge as $a$ with the opposite orientation then $\alpha (a)^{-1} =  \alpha (\bar{a})$
\end{itemize}
\end{defn}

\begin{remark}
We can extend $\alpha$ to a map from all simplicial paths $\gamma = e_1 \ldots e_n$ by defining
\[ \alpha(\gamma) = \alpha (e_1) \ldots \alpha(e_n) \]
and then because every homotopy of simplicial paths can be realised as a sequence of homotopies across 2-simplices, $\alpha$ is invariant under homotopy fixing end points. This means that $\alpha$ induces a homomorphism $\alpha_* : \pi_1(M) \rightarrow \Lambda$.
\end{remark}

\begin{defn}[Semi-ideal $n$-simplex]
A semi-ideal $n$-simplex is one which has had one of its vertices removed. We call this vertex ideal. 
\end{defn}

\begin{defn}[Semi-Ideal Triangulation]
Let $M$ be a finite volume hyperbolic $n$-manifold, then a semi-ideal triangulation of $M$ is a simplicial complex where we remove some of the vertices before defining a homeomorphism to $M$.
\end{defn}

\begin{convention}
In the case of a semi-ideal triangulation of a manifold $M$, a cocycle $\alpha \in C^1(M,\Lambda)$ shall take as inputs only the non-ideal edges of the triangulation. This is justified as the semi-ideal triangulation deformation retracts to its non-ideal simplices.
\end{convention}

\begin{defn}[G-equivariant]
Let $X, Y$ be $G$-sets for $G$ a group. Let $F: X \rightarrow Y$ be a map such that
\[ \forall g \in G, \forall x \in X, \qquad F(g\cdot x) = g \cdot F(x) \]
Then we say  $F$ is $G$-equivariant and $F$ induces a map $\hat{F} : G \backslash X \rightarrow G \backslash Y$.
\end{defn}

\begin{defn}
Let $\tilde{M}$ be the universal cover of $M$ and let $\phi$ be the canonical isomorphism $\pi_1(M, x_0) \rightarrow Deck(\tilde{M}, \tilde{x_0})$ defined by picking as a base point the lift $\tilde{x_0}$ of $x_0$, then for each $[\gamma] \in \pi_1(M, x_0)$ we define $g_\gamma : = \phi([\gamma])$ .
\end{defn}



Let $M$ be a $n$-manifold admitting a finite volume hyperbolic structure equipped with a triangulation (semi-ideal if the manifold is not closed) and let $X = G \backslash\h^n $  be a hyperbolic $n$-manifold homeomorphic to $M$, so $G \leq Isom^+(\h^n)$. Let $f$ be a homeomorphism $f : M \rightarrow X$, then $f$ induces a $G$-equivariant homeomorphism $\tilde{f}: \tilde{M} \rightarrow \h^n$ because the action of $G$ on $\tilde{M}$ is defined by following the chain of isomorphisms 
\[ G \cong \pi_1(X) \xleftarrow{f_*}\pi_1(M) \cong Deck(\tilde{M}).\]

\begin{lemma}\label{HomotopyEquiv}
Let $M, X, G$ be as in the preceeding paragraph. Let $\alpha$ be a cocycle such that $\alpha_* : \pi_1(M, x_0) \rightarrow Isom^+(\h^n)$ is a representation of $\pi_1(M)$ with Im$(\alpha_*) = G$. Then we can define a map $\tilde{F} : \tilde{M} \rightarrow \h^n$ such that for each (possibly semi-ideal) $n$-simplex $\Delta$ in $\tilde{M}$, $\tilde{F} (\Delta)$ is a geodesic $n$-simplex (it is the convex hull of its vertices) and $\tilde{F}$ induces a homotopy equivalence $F : M \rightarrow X$. Furthermore, the map which $F$ induces on fundamental groups is exactly $\alpha_*$.
\end{lemma}

\begin{proof}
%
First suppose that $M$ is closed. 
Let $\tilde{x_0} \in \tilde{M}$ be some lift of $x_0$ and take $\tilde{F}(\tilde{x_0}) = \tilde{y_0}$ some arbitrary point in $\h^n$. Now let $\tilde{x}$ be any other vertex in the $0$-skeleton of $\tilde{M}$ and $\tilde{\gamma}$ a simplicial path from $\tilde{x_0}$ to $\tilde{x}$. Let $\gamma$ be the projection of $\tilde{\gamma}$ to $M$, then define $\tilde{F}(\tilde{x}) = \alpha(\gamma)\cdot \tilde{y_0}.$ 
Assuming that $\tilde{F}$ is $G$-equivariant on the vertices, we define $\tilde{F}$ dimension by dimension, first we map the 1-simplices parameterised by arc-length onto the geodesic between the image of their vertices. For higher dimensional simplices, pick one representative of each $G$-orbit, and define inductively on dimension by picking for each $k$-simplex representative some homeomorphism onto the convex hull of the image of its vertices which is an extension of the homeomorphism defined inductively on its boundary. We then define $\tilde{F}$ on all the translates of our simplex by $\tilde{F} (g(x)) = \alpha_*(g) \tilde{F}(x)$ for all $x$ in our chosen representative.

By construction, $\tilde{F}$ is $G$-equivariant iff the map on the $0$-skeleton is. Let $\tilde{x}$ be as above, then $\tilde{x} = g_\gamma (\tilde{x_0})$. Take $g_{\gamma'}$ some other element of Deck$(\tilde{M})$, then 
\[ \tilde{F}(g_{\gamma'} (\tilde{x})) = \tilde{F} (g_{\gamma'} g_\gamma(\tilde{x_0})) =\alpha(\gamma')\alpha(\gamma) \tilde{F}(x_0) = \alpha(\gamma')\tilde{F} (g_\gamma(\tilde{x_0})) = \alpha(\gamma')\tilde{F} (\tilde{x}) . \]
Hence, $\tilde{F}$ is $G$-equivariant.

Note here that if we take some homeomorphism from $M$ to $X$, then this lifts to a map $\tilde{F}'$ from $\tilde{M}$ to $\tilde{X}$ and by point-pushing $G$-equivariantly, we can assume that this map agrees with $\tilde{F}$ on the vertices. Because of this we can define a straight line homotopy on each simplex of $\tilde{M}$ which commutes with the $G$-action and takes $\tilde{F}$ to $\tilde{F}'$. Thus $F$ and $F'$ are homotopic, and as $F'$ is a homeomorphism, $F$ must be a homotopy equivalence.


In the cusped case, we can use the same argument if we can define $\tilde{F}$ on the ideal vertices in $\tilde{M}$  and show it to be $G$-equivariant on them too. Note here that  we are in fact defining $\tilde{F}$ by defining a map from $\tilde{M} \cup \{ \text{ideal vertices} \}$ to $\h^n \cup \partial \h^n$ and then taking its restriction to $\tilde{M}$. However, we will abuse notation to refer to both maps as $\tilde{F}$ throughout.  We number the cusps and take a spanning tree of the $1$-skeleton of the link of the ideal vertex $v_i$ corresponding to the $i$th cusp and add some edge which ends at $v_i$. We also choose a path $\delta$ in the 1-skeleton of $M$ from our basepoint to this tree such that the union is still a tree, call this union $\Gamma_i$. 
Any edge in the link of the ideal vertex which is outside of the spanning tree gives a loop $\sigma$ by adjoining the unique geodesic in $\Gamma_i$ linking its endpoints. We can make $\sigma$ be based at $x_0$ by adjoining our path $\delta$ at both ends. As $\sigma$ is homotopic into the ideal vertex, $\alpha(\sigma)$ is thus a parabolic element of $G$, fixing some point $p$ in $\partial \h^n$. 

If $\tilde{v_i}$ is the lift of $v_i$ which lies in the same lift of $\Gamma_i$ as $\tilde{x_0}$, then we define $\tilde{F} (\tilde{v_i}) = p$. Each lift $\tilde{v_i} $ of $v_i$ lies in a lift of $\Gamma_i$ so if we take some path $\gamma$ from $\tilde{x_0}$ to this new lift of $\Gamma_i$, then we define $\tilde{F}(\tilde{v_i}) =  \alpha(\gamma) p$.
From here, we only need to check two things, first that $\tilde{F}$ is well defined, and second that it is $G$-equivariant. 
For the well definedness, note that for a lift $\tilde{v_i}$, there are many lifts of $\Gamma_i$ to which it might belong and thus distinct choices, $\gamma, \gamma'$ for the path defined above. However all such pairs of lifts of 	$\Gamma_i$ can be connected by a path which lies entirely in the link of the lift and maps down to some loop $\delta$ which is homotopic into $v_i$. Thus $\alpha(\delta)$ corresponds to a parabolic element fixing $p$ and so as $\gamma' = \gamma \delta$, we get that 
\[   \alpha(\gamma') p = \alpha(\gamma)\alpha(\delta)p = \alpha(\gamma)p  \]
so $\tilde{F}$ is well defined.
For the $G$-equivariance, note that for $g$ some deck transformation $\tilde{x_0}$ and $g(\tilde{x_0})$ are linked by a path $\gamma$ with $\alpha(\gamma) = g$ so as $g \tilde{v_i}$ lies in the same lift of $\Gamma$ as $g\tilde{x_0}$ we have that 
\[  \tilde{F}(g \tilde{v_i}) =  \alpha(\gamma) p= \alpha(\gamma) \tilde{F} (\tilde{v_i}) = g \tilde{F} (\tilde{v_i})\]
and hence for any $h \in G$ and any ideal vertex $\tilde{v} = g\tilde{v_i} \in \tilde{M}$, the following holds
\[ \tilde{F}(h \tilde{v} ) = \tilde{F} (hg \tilde{v_i} )=  hg \tilde{F} (\tilde{v_i} ) = h \tilde{F} ( \tilde{v})\]
\end{proof}

\begin{lemma}\label{HomEqSurj}
A homotopy equivalence $f:M \rightarrow N$ between closed orientable $n$-manifolds is a surjection.
\end{lemma}

\begin{proof}
Note that as $f$ is a homotopy equivalence, it induces an isomorphism on $n$th homology groups, sending a fundamental class $[M]$ to a fundamental class $[N]$.\\
Now suppose $f$ is not surjective, and let $y \in N$ be outside the image of $f$. Then consider the natural isomorphism $H_n(N) \cong H_n(N, N-y)$ then for any representative of $[M]$, $f[M]$ lies entirely in $N-y$ and so $f[M]$ is trivial in $H_n(N, N-y)$	 and thus trivial in $H_n(N)$ contradicting that $f$ is a homotopy equivalence.
\end{proof}

\begin{remark}
It is also possible (see Lemma \ref{CuspedSurjection}) to prove that the homotopy equivalence defined in Lemma \ref{HomotopyEquiv} is surjective in the non-compact case, but that requires a more nuanced argument as homotopy equivalences between non-compact manifolds can easily not be surjections, for example consider a standard map from $\R^n$ to itself which maps everything to the point.
\end{remark}

We will soon see how to find cocycles as the solutions of a system of polynomials, but these cocycles could be trivial or induce non faithful representations, so to use the results of Section \ref{SecPolynomial} we need some connected component of the solution set to consist of cocycles which all induce faithful lattice representations (see below definition) which define homotopy equivalences by the method of Lemma \ref{HomotopyEquiv}. We shall prove the existence of this connected component using Calabi-Weil Rigidity.


\begin{defn}
Let $G$ be a group, $\Lambda =  Isom(\h^n)$ and let $H : = \{ H_1, \ldots H_q \}$ a collection of subgroups of $G$. Define
\[ Hom_{par}(G, H ; \Lambda) := \{ \rho :G \rightarrow \Lambda \mid  \rho (\gamma) \text{ is parabolic for all } \gamma \in H_j, j=1, \ldots q\} \] 
If $G$ is a subgroup of $\Lambda$, then we define
\[ Hom_{par}(G ; \Lambda) :=  Hom_{par}(G, H_{\text{max}} ; \Lambda) \]
where $H_{\text{max}}$ is a collection of representatives of the conjugacy classes of maximal parabolic subgroups of $G$. Note that if $G$ is a uniform lattice, then $ Hom_{par}(G ; \Lambda) =  Hom(G ; \Lambda)$.

\end{defn}
\begin{theorem}[Calabi-Weil Rigidity, Thms. 4.19 and 8.55 in \cite{Kapovich00} ]

Let $G$ be a  lattice in $Isom(\h^n)$ where $n\geq 3$. The identity representation $\rho_0: G \rightarrow Isom(\h^n):=\Lambda$ lies in a connected component of $ Hom_{par}(G; \Lambda)$ made up entirely of conjugates of $\rho_0$.
\end{theorem}

For the sake of comparison we place here the similar result of \cite{GarlandRaghunathan70} which we shall use for the cusped case in section \ref{SecCusp}. It crucially doesn't cover $3$-dimensional manifolds so we can't use it in the closed case, but it avoids the requirement on parabolics in the cusped case.

\begin{reptheorem}{GarlandRaghunathan}[Theorem 7.2 in \cite{GarlandRaghunathan70}]
Let $G$ be a  lattice in $Isom(\h^n)$ for $n \geq 4$. The identity representation $\rho_0: G \rightarrow Isom(\h^n):=\Lambda$ lies in a connected component of $ Hom(G; \Lambda)$ made up entirely of conjugates of $\rho_0$.
\end{reptheorem}

\begin{defn}[Lattice Representation]
We call a representation $G \rightarrow Isom(\h^n)$ a lattice representation if its image is a lattice (a discrete co-finite volume subgroup).
\end{defn}

To round out our rigidity results, we recall Mostow-Prasad rigidity, which tells us (among other things) that all faithful lattice representations of a group induce isometric hyperbolic structures, a fact that we shall be implicitly using throughout the rest of the paper.

\begin{theorem}[Mostow-Prasad Rigidity \cite{Prasad73}]
Let $\Gamma, \Gamma'$ be two lattices in $Isom(\h^n)$ and $\phi : \Gamma \rightarrow \Gamma'$ an isomorphism .
Then there exists an isometry $g$ in $Isom(\h^n)$ such that 
\[ g\gamma g^{-1} = \phi(\gamma) \quad \forall \gamma \in \Gamma \]

\end{theorem}
In the following lemma we use the isomorphism $SO(n,1) \cong Isom^+(\h^n)$ coming from the hyperboloid model of hyperbolic $n$-space.
\begin{lemma} \label{cocyclepolys1}
Let $M$ be a closed hyperbolic $n$-manifold triangulated by $t$ $n$-simplices. There is a system of polynomial inequalities as in Theorem \ref{Grigoriev} such that the following all hold:

\begin{itemize}
\item The solution set consists of cocycles in $C^1(M, SO(n,1))$.
\item The solution set has a component consisting entirely of cocycles each of which induces a  faithful lattice representation.
\item The system consists of $\leq (n+1)^5t$ polynomials, in $\leq (n+1)^4 t $ variables with degree $\leq 2$ and coefficients all $\pm 1$.
\end{itemize} 
\end{lemma} 

\begin{proof}
For each edge $e$ of the triangulation of $M$ (of which there are less than ${n+1 \choose 2} t$), define two sets of $(n+1)^2$ variables corresponding to the entries of a matrix in $SO(n,1)$, these variables correspond to the images of both orientations of $e$ under some cocycle $\alpha$. Then, for each dimension 2 face of the triangulation (of which there are less than ${n +1 \choose 3}t$) we get a relation which can be expressed as $(n+1)^2$ polynomials (for the $(n+1)^2$ variables defining the matrix) which are satisfied iff the matrices corresponding to the first two edges of the face multiply to give the matrix corresponding to the third edge, these being the face relations of the cocycle.
We also include polynomials which ensure that opposite orientations of an edge are sent to inverse elements, this consists of $(n+1)^2$ polynomials for each edge, and these polynomials have coefficients $\pm 1$ and are quadratic. 

Thus we have a system of $\leq (n+1)^5t$ polynomials in $\leq (n+1)^5t$ variables with coefficients always 1 or -1 and which are at most quadratic. Furthermore, solutions to this system correspond to maps from the oriented $1$-simplices of $M$ to $SO(n,1)$ which satisfy the face relations and orientation reversing relations, thus by definition solutions correspond to $1$-cocycles.

As our manifold $M$ admits a hyperbolic structure, we know that it admits a faithful lattice representation and thus there exists a cocycle which induces such a representation. Now any continuous deformation of this cocycle induces a continuous deformation of the induced representation. Thus Calabi-Weil Rigidity implies that any component containing a cocycle which induces a  faithful lattice representation consists entirely of such cocycles.
\end{proof}

At this point we could apply the results of Section \ref{SecPolynomial} and get bounds on the size of the matrices in our cocycle and use that to find bounds on the geometry of $M$. Instead, we're going to expand the system of polynomials so that the geometric quantities we want to bound appear among the variables and thus the results of Section \ref{SecPolynomial} immediately give us the desired bounds. 

For the following Lemma it is useful to note that by finiteness of the triangulation, any cocycle can be altered (by multiplying wherever necessary the value on all edges round a vertex by a sufficiently small isometry) so as to induce the same representation, while taking non-zero values on every edge. Thus we don't lose any representations by making this requirement of our cocycles.

\begin{lemma} \label{cocyclepolys2}
We can expand the system of polynomials above to include variables corresponding to the lengths of the edges in the image of the homotopy equivalence $F$ defined in Lemma \ref{HomotopyEquiv}. We also add the requirement that each edge be non-zero length.

The solutions of this new system are in one to one correspondence with the subset of the old solution set consisting of all cocycles which take non-zero values on all edges.
\end{lemma}

\begin{proof}
In the proof of Lemma \ref{HomotopyEquiv} we didn't specify a model for the universal cover of either manifold, though both have universal cover $\h^n$. Here we shall take the hyperboloid model as our model and we shall take $(0,\ldots,0,1)$ as our basepoint. Now $SO(n,1)$ acts on the hyperboloid by matrix multiplication.
Thus the map on universal covers $\tilde{F} : (\tilde{M}, \tilde{x_0}) \rightarrow (\h^n, (0, \ldots ,0,1))$ is defined by sending the point at the end of the path $\gamma$ in $\tilde{M}$ to the point $A_\gamma  (0,\ldots, 0,1)^T $ where $\alpha(\gamma) =: A_\gamma \in SO(n,1)$.
Thus the vertices of the triangulation of the universal cover are given by $A (0,\ldots, 0,1)^T$ where $A$ is $A_\gamma$ for some $\gamma$.

As there are at most ${n+1 \choose 2}t$ edges in $M$ we can find a lift $\tilde{e}$ of each edge $e$ of $M$ such that both vertices $\underline{x}, \underline{y}$ of $\tilde{e}$ lie within ${n+1 \choose 2}t$ edges of the basepoint. Picking paths $\gamma_1, \gamma_2$ of length $\leq {n+1 \choose 2}t$ to each vertex of $\tilde{e}$, yields group elements $A_{\gamma_1}, A_{\gamma_2}$ which can be expressed as $(n+1)^2$ polynomials of degree at most ${n+1 \choose 2}t$ in the original variables (each $A_\gamma$ is a product of $A_{e_i}$ with $e_i$ edges of the triangulation). 

We build this up edge by edge, at each point we define at most one new vertex for each new edge as a set of $n+1$ new variables. In total we need $\leq (n+1){n+1 \choose2} t$ new variables and polynomials defining these vertices. The length of the edge between two such vertices $\underline{x}$ and $\underline{y}$ where  $\underline{x} = (x_0,\ldots, x_{n})$ and $\underline{y} = (y_0,\ldots, y_{n})$can then be calculated as 
\[l(e) = d(\underline{x},\underline{y}) = arcosh ( x_0y_0 + \ldots + x_{n-1}y_{n-1} - x_ny_n).\] This is not a polynomial and so we instead define a variable for each edge as 
\[C = cosh(l(e)) -1= cosh(d(\underline{x}, \underline{y})) -1\]
which is certainly a polynomial, the number of variables $C$ is the number of edges which is $\leq{n+1 \choose2} t$. We then require that $C > 0$, and thus that $l(e)>0$.

In total we have added less than $(n+2){n+1 \choose2} t$ new variables and $(n+2){n+1 \choose2} t$ polynomials with degree bounded by ${n+1 \choose 2}t$ and coefficient complexity still at 1.
\end{proof}

Combining the above two proofs gives a polynomial system which has size described (in the language of Section \ref{SecPolynomial}) by 
\begin{itemize}
\item $ N \leq (n+1)^4t +(n+2){n+1 \choose2} t  \leq (n+2)^4t$
\item $\kappa \leq (n+1)^5t + (n+2){n+1 \choose2} t  \leq (n+2)^5t$
\item $d \leq {n+1 \choose 2}t \leq (n+1)^2t$
\item $M \leq 2$
\end{itemize}

Thus, as the results of Section \ref{SecPolynomial} bound the size of non-zero variables the following theorem is a simple corollary of Corollary \ref{variablebound} and Lemmas \ref{HomotopyEquiv}, \ref{cocyclepolys1}, and \ref{cocyclepolys2}. 

\begin{theorem}\label{EdgeBound}
Let $M$ be a closed hyperbolic manifold, there exists a homotopy equivalence $F$ as described in Lemma \ref{HomotopyEquiv} such that the image of each edge has length $l(e)$ bounded in the following way:
\[ cosh(l(e)) \leq 2^{2 (n+2)^5t)^{O((n+2)^4t)}}  \]
This can be simplified, due to big O notation to the following bound:
\[  cosh(l(e)) \leq 2^{(nt)^{O(n^4t)}} \]
And hence, that 
\[ 0 < l(e) \leq   (nt)^{O(n^4t)} \]
\end{theorem}

\begin{proof}
Lemmas \ref{cocyclepolys1} and \ref{cocyclepolys2} provide us with a system of polynomials, the solution set of which has a connected component, all elements of which induce faithful lattice representations. Mostow Rigidity tells us each of these lattices defines a hyperbolic manifold $X$ as in the setup of Lemma \ref{HomotopyEquiv}. Thus taking a bounded solution given by Theorem \ref{Grigoriev} means that all the non-zero variables of this solution are bounded as in the statement of the Theorem. In particular, as edge lengths are all non-zero the result follows.
\end{proof}

\section{Diameter of Margulis Tubes} \label{SecInjRad}

In this section we shall describe some results of hyperbolic geometry which will allow us to bound the length of a systole from below as a function of the diameter of the Margulis tube containing it. This will be used in conjuction with our bound on edge lengths from the previous section to bound the systole length of a hyperbolic manifold $M$ from below as a function of the number of tetrahedra in any topological triangulation of $M$.

The tool we will use to do this is the thick-thin decomposition, which we shall briefly recall here, following the books of Benedetti and Petronio and Jessica Purcell  \cite{Benedetti1992}\cite{Purcell20}.

\begin{remark}
The model of hyperbolic space we're using in this section is different from in the previous section. For the purpose of understanding cocycles and defining a system of polynomials, we need to be working with a variety so we used the hyperboloid model because its isometry group is a variety. In this section, we will use the Euclidean structure of horospheres and so we use the upper half space model where any given horosphere can be taken to a horizontal Euclidean plane by an isometry of $\h^n$.
\end{remark}

\begin{defn}[Injectivity Radius]
We define the injectivity radius at a point $x \in M$ to be
\[ \text{injrad}(x) = sup \{ r \mid B(x,r) \text{ is an embedded r-ball in } M  \}. \]
The injectivity radius of a manifold is then given by 
\[\text{injrad} (M) = \text{inf} \{ \text{injrad}(x) \mid x \in M \}.\]
Note that in the case where $M$ is closed, there is some $x \in M$ such that $\text{injrad}(x) = \text{injrad}(M)$.
\end{defn}

\begin{defn}
A systole of a hyperbolic manifold $M$ is a geodesic $\gamma$ in $M$ of shortest length. 
\end{defn}

\begin{defn}
We define the $\epsilon$-thin part $M_{(0,\epsilon)}$ of a hyperbolic manifold $M$ to be 
\[ M_{(0,\epsilon)} = \{ x \in M \mid \text{injrad}(x) < \frac{\epsilon}{2} \} .\]
Similarly the $\epsilon$-thick part is defined to be 
\[  M_{[\epsilon, \infty )}  = \{ x \in M \mid \text{injrad}(x) \geq \frac{\epsilon}{2} \} \]

\end{defn}

\begin{theorem}[Margulis' Lemma, D3.3 in Benedetti and Petronio \cite{Benedetti1992}] \label{Margulis}

There exists a universal constant $\epsilon_n$ for each $n \geq 3$ such that the following holds.

Let $M$ be a complete oriented hyperbolic $n$-manifold (not necessarily compact or finite volume). The thin part $M_{(0,\epsilon_n)}$ is the union of pieces homeomorphic to one of the following types:

\begin{enumerate}
\item $\mathring{D}^{n-1} \times S^1$
\item $V \times (0,\infty)$ where V is a differentiable oriented $(n-1)$-manifold without boundary supporting a Euclidean structure.
\end{enumerate}
Furthermore,
\begin{itemize}
\item The pieces are a positive distance from one another.
\item The second case occurs only when the manifold is non-compact.
\item When $M$ is finite volume non-compact, the second case occurs and the manifolds $V$ are closed.
\end{itemize}
\end{theorem}

\begin{defn}
We shall refer to the two types of component of the thin part as tubes (also called Margulis tubes) and cusp neighbourhoods.
\end{defn}
Note that if a systole of a hyperbolic $n$-manifold  lies in the thin part of $M$, it must lie in one of these tubes. Otherwise it lies in a cusp neighbourhood, and could be homotoped further towards the cusp to reduce its length. 

 We now consider what these Margulis tubes are like if a systole of the manifold is very short.
First a quick lemma about closest point projection to the vertical axis in the upper half space model.

\begin{lemma} \label{DistanceFromAxis}
The distance from a point $\underline{x}$ to the vertical axis through the origin in the upper half space model is given by
\[ arcosh \left( \frac{\lVert \underline{x} \rVert}{x_n} \right) \]
where $\lVert \underline{x} \rVert$ is the standard Euclidean norm.
\end{lemma}

\begin{proof}

Let $\underline{x}:=(x_1, \ldots x_n)$ be a point in the upper half space model and note that the closest point projection of $\underline{x}$ to the vertical axis is $ \underline{y} := (0, \ldots,0, \lVert \underline{x} \rVert )$ as both points lie on a half circle with centre on the boundary at infinity which intersects the vertical axis at a right angle.

Now a formula for distance in the upper half space model is given by 

\[d( \underline{x}, \underline{y}) = arcosh\left( 1+ \frac{\sum_{i=1}^{n} (x_i - y_i)^2}{2x_ny_n} \right) \]

So substituting in our values gives us 

\[  d( \underline{x}, \underline{y}) 
= arcosh\left( 1+ \frac{\left( \sum_{i=1}^{n-1} (x_i)^2 \right) + (x_n -  \lVert \underline{x}\rVert)^2}{2x_n\lVert \underline{x} \rVert} \right)  \]

\[= arcosh\left( 1+ \frac{ 2 \lVert \underline{x} \rVert^2  - 2x_n \lVert \underline{x}\rVert}{2x_n\lVert \underline{x} \rVert} \right) = arcosh\left(\frac{  \lVert \underline{x} \rVert  }{x_n} \right)
\]
\end{proof}

\begin{lemma} \label{RotationDistance}

Let $\phi$ be a loxodromic isometry in $Isom(\h^n)$ acting on the upper half space model with fixed points at $0$ and $\infty$. Suppose the translation length of $\phi$ along its axis is $R$. If $\underline{x}$ is such that $\frac{\lVert \underline{x} \rVert}{x_n} \leq e^D$, then for each $1> a >0$, we can find an integer $0< k \leq \left( \frac{4e^{D}}{a} \right)^{n-1} $ s.t. $ d \left (\phi^k(\underline{x}), e^{kR}(\underline{x})\right )< a$.

\end{lemma}

\begin{proof}

A loxodromic isometry $\phi$ of translation length $R$ acting on the  upper half space model of hyperbolic space fixing $0$ and $\infty$ is a composition of a scaling $\underline{x} \mapsto e^R \underline{x}$ and an orthogonal map $A \in SO(n-1)$ acting only on the first $n-1$ variables. Thus if $ \underline{x} = (x_1, \ldots x_n)$ then 
\[d(\phi^k(\underline{x}), e^{kR}\underline{x}) = d(A^k \cdot e^{kR}(\underline{x}), e^{kR}\underline{x}) = d(A^k \underline{x}, \underline{x}) \]

We can apply a further isometry such that $\underline{x}, A^k\underline{x}$ both lie on the plane $x_n=1$ and by Lemma \ref{DistanceFromAxis} it is still true that $\frac{\lVert \underline{x} \rVert}{x_n} \leq e^D$ and hence if we let $\pi$ be the projection map onto the first $n-1$ factors we get
\[e^{2D}\geq \lVert \underline{x} \rVert^2 = \lVert \pi(\underline{x})\rVert^2 +1 \qquad \text{ and so } \qquad \sqrt{e^{2D} -1 } \geq  \lVert \pi(\underline{x}) \rVert. \]

In particular, if we restrict to the hyperplane $\R^{n-1} \times \{ 1 \}$ the $\frac{a}{2}$-ball around each of the translates $A^k \underline{x}$ lies within the $2e^{D}$-ball in $\R^{n-1} \times {1}$. Indeed, let $\underline{x}'$ lie in such an $\frac{a}{2}$-ball, then as $a < 1$:
\[ \lVert \pi(\underline{x}') \rVert^2 \leq \lVert \pi(\underline{x}) \rVert^2 +a\lVert \pi(\underline{x}) \rVert  +\frac{a}{2} < e^{2D} - 1 + a \sqrt{e^{2D} -1 } + \frac{a}{2} < 2e^D \]
Now let $vol (B(r)) $ denote the volume of the $r$-ball in Euclidean $(n-1)$-space and define
\[ f(a) = \frac{ vol (B(2e^{D}))}{vol(B(\frac{a}{2}))} .\]
If $k> f(a)$ then there must be some pair of points $A^i\underline{x}, A^j\underline{x}$ with $i<j$ such that their $\frac{a}{2}$-balls overlap, otherwise all the balls are disjoint but the sum of their volumes is greater than the volume of the $2e^D$-ball containing them.

Thus the result follows as 
\[\frac{ vol (B(2e^D))}{vol(B(\frac{a}{2}))} = \frac{(2e^D)^{n-1}}{\left(\frac{a}{2}\right)^{n-1}} \frac{vol(B(1))}{vol(B(1))} =\left(\frac{(4e^{D})}{a}\right)^{n-1} \]
and if the neighbouring balls of $A^i\underline{x}, A^j\underline{x}$ overlap then so do the balls round $\underline{x}, A^{j-i}\underline{x}$ and so $A^{j-i}\underline{x}$ is within Euclidean  (and hence also hyperbolic) distance $a$ of $\underline{x}$ as required.
\end{proof}

\begin{theorem} \label{InjRadDiameter}
Suppose that $M$ is a finite volume hyperbolic $n$-manifold  $(n \geq 3)$ with  systole(s) of length $R \leq 2\epsilon_n$, where $\epsilon_n$ is the Margulis constant in dimension $n$.
Then the distance from a systole to the boundary of the Margulis tube containing it is bounded below by \[\frac{1}{n} log\left(\frac{1}{R}\right) +  log (\epsilon_n) -log( 4). \]
\end{theorem}

\begin{proof}
Let $\gamma$ be a systole of $M$. We consider $N = N(\gamma, D)$ in an attempt to  bound below the maximum value of $D$ such that this neighbourhood lies in the thin part of $M$. 

By definition $\gamma$ is a geodesic and so the corresponding isometry of hyperbolic space in the deck group of $M$ is a loxodromic with geodesic axis $\tilde{\gamma}$ which projects down to $\gamma$ under the covering map. By conjugation this can be modelled in the upper half space model of hyperbolic space by the isometry 
\[ \phi: \underline{v} \mapsto A \cdot e^R\underline{v} \]
with geodesic axis the geodesic from $0$ to $\infty$ where $R$ is the translation length of $\phi$ and $A$ is an orthogonal map on the first $n-1$ variables.

Now let $\underline{x}: = (x_1, \ldots, x_{n})$ be a point in $\h^3$ within distance $D$ from the axis of $\phi$, so in particular $\frac{\lVert \underline{x} \rVert}{x_n} <  cosh (D) < e^D $ and let $\underline{x}' := (x_1, \ldots, x_{n-1}, e^{kR}x_n)$ be the vertical translate of $\underline{x}$ by hyperbolic distance $kR$. Then 
\[ d(\underline{x},\phi^k(\underline{x})) \leq d(\underline{x},\underline{x}') + d(\underline{x}', e^{kR}\underline{x}) + d(e^{kR}\underline{x}, \phi^k(\underline{x})).\]

Now $d(\underline{x},\underline{x}')$ is just the length of the vertical geodesic from $x$ to $x'$. This is $log(e^{kR}x_n) - log(x_n) = log(e^kR)=kR$. 

Furthermore, the distance from $\underline{x}'$ to $e^{kR}\underline{x}$ is bounded above by their distance in the horosphere, which is Euclidean up to scaling. So if $\pi(x)$ represents the projection to the first $n-1$ factors, the following holds.
\[ d(\underline{x}',e^{kR}(\underline{x})) \leq \frac{1}{e^{kR}x_n} \lVert e^{kR} \pi(x)- \pi(x) \rVert = \frac{(e^{kR}-1) \lVert \pi(x) \rVert}{e^{kR}x_n} \leq (e^{kR}-1)(e^D) \]
and so by Lemma \ref{RotationDistance}
\[ d(x,\phi(x)) \leq kR + (e^{kR}-1)e^D + a.  \]


At this point we note that the following statement follows by simple rearrangment and taking logarithms

\begin{equation}
  D < \frac{1}{n} log\left(\frac{1}{R}\right) +  log (\epsilon_n) -log( 4)\quad \Leftrightarrow \quad R\left( \frac{4e^{D}}{\epsilon_n} \right)^{n} <   1 
\end{equation}


Thus, if we set $a = \epsilon_n$ and require
\[ D < \frac{1}{n} log\left(\frac{1}{R}\right) +  log (\epsilon_n) -log( 4)\]
then for $k <   \left( \frac{4e^{D}}{a} \right)^{n-1} $
it is both true that 
\[ kR = R\left( \frac{4e^{D}}{\epsilon_n} \right)^{n-1} < R\left( \frac{4e^{D}}{\epsilon_n} \right)^{n} <1  \] 
and
\[ kR(2e^D+1) <  \left( \frac{4e^{D}}{\epsilon_n} \right)^{n-1} R(4e^D) < R \left( \frac{4e^{D}}{\epsilon_n}\right) ^{n}\epsilon_n < \epsilon_n \] 

Combining this with the fact that $e^x-1 < 2x$ on the interval $(0,1)$ gives that
\[ d(\underline{x}, \phi(\underline{x})) \leq kR + (e^{kR}-1)e^D + a \leq ( 1 + 2e^D)kR  + \epsilon_n < 2\epsilon_n \]
and hence if $\underline{x}$ is within distance $D$ of the axis of $\phi$ then $\underline{x}$ maps to the thin part of $M$. Thus the thick-thin decomposition implies that the diameter of the Margulis tube is at least
\[\frac{1}{n} log\left(\frac{1}{R}\right) +  log (\epsilon_n) -log( 4) \] 
\end{proof}

If our desire is to give an explicit relation bound on injectivity radius and diameter as a function of dimension, and number of $n$-simplices, then we need to understand $\epsilon_n$ as a function of $n$. Fortunately, this an area of much research and lower bounds are known.  Robert Meyerhoff gives a bound on $\epsilon_3$ \cite{Meyerhoff87}. For general $n$, a result of Ruth Kellerhals \cite{Kellerhalls04} gives a  lower bound for $\epsilon_n$.

\begin{theorem}[Meyerhoff \cite{Meyerhoff87}]\label{Meyerhoff}
With the definition of the thick-thin decomposition given in Theorem \ref{Margulis}, the value of $\epsilon_3$ is at least $0.052$ .
\end{theorem}

\begin{theorem}[Kellerhalls \cite{Kellerhalls04}]\label{Kellerhalls}

The value of $\epsilon_n$ is bounded below by $\frac{1}{(6\pi)^n}$.

\end{theorem}

In fact Kellerhals' paper gives a tighter lower bound than this, but we shall use this coarser bound for simplicity of expression.

\begin{theorem} \label{MainThm2}
Given a closed hyperbolic $n$-manifold $(n \geq 3)$ triangulated by $t$ \linebreak $n$-simplices, let $R$ be the length of a systole in $M$, then
\[ R \geq \frac{1}{2^{(nt)^{O(n^4t)} }}. \]
\end{theorem}

\begin{proof}

Let $X$ be the hyperbolic manifold defined in Lemma \ref{HomotopyEquiv}, then by Lemma \ref{EdgeBound} we know that $X$ admits a cover by $t$ immersed hyperbolic $n$-simplices with edge lengths bounded above by some bound $B(t)$ with
\[ B(t) \leq (nt)^{O(n^4t)}. \] 
Pick some $\epsilon > \frac{1}{(6\pi)^n}$ so that $M_{(0,\epsilon)}$  satisfies the conclusion of Theorem \ref{Margulis}. This is possible as the bound in Theorem \ref{Kellerhalls} is not strict.

If the thin part $M_{(0,\epsilon)}$ is empty, then $R \geq \frac{2}{(6\pi)^n}$ and so the conclusion holds. \\
Suppose the thin part is non empty, then Theorem \ref{InjRadDiameter} applies. 
Consider a systole $\gamma$. This curve has length $R$ and lies in the centre of one of the tubes of the thin part. We know from the above that the diameter of this tube is at least
\[ \frac{1}{n} log(1/R) +log (\frac{\epsilon_n}{4}). \]

However, we know that $M$ is covered by $t$ $n$-simplices of bounded edge length, and thus of bounded diameter, and so the diameter of $M$ is at most 
\[t B(t) \leq t(nt)^{O(n^4t)} \]
and hence

\[ \frac{1}{n} log(1/R) +log (\frac{\epsilon_n}{4}) \leq t(nt)^{O(n^4t)} .\]
Simplifying big O notation further and applying Theorem \ref{Kellerhalls}
\[ log\left(\frac{1}{R}\right) \leq   (nt)^{O(n^4t)} + nlog( 4(6\pi)^n)   .\]
Taking exponentials on both sides gives
\[\frac{1}{R} \leq e ^{(nt)^{O(n^4t)}} (24\pi)^{n^2}  \] which simplifies, after rearrangement, to

\[ R \geq \frac{1}{2^{(nt)^{O(n^4t)} }} \]
\end{proof}

\section{The Cusped Case} \label{SecCusp}

With some small but technical adaptations, the arguments from the closed case hold in the cusped (finite volume) case where $M$ is given by a finite semi-ideal triangulation and $M$ admits a finite volume hyperbolic structure.

\begin{lemma} \label{CuspedSurjection}
The homotopy equivalence defined in Lemma \ref{HomotopyEquiv} is surjective also in the finite-volume case.
\end{lemma}

\begin{proof}
In the proof of Lemma \ref{HomotopyEquiv} we showed that the homotopy equivalence $F$ is homotopic to a homeomorphism $F'$, and in fact the homotopy fixes ideal vertices, so the entire homotopy extends to the space we get when we add the ideal vertices back in (replace all ideal simplices with non-ideal ones). Call this pseudomanifold $M'$ and the target $X'$. Then the $n$th homology of $M', X'$ is $\Z$ with fundamental class $[ M' ]$ given by the triangulation formed by replacing the ideal vertices with concrete ones.  Now note that $[M']$ generates both $H_n(M')$ and $H_n(M', M'-x)$ for $x \in M$, and so, as $F[M']$ is homotopic and thus homologous to $F'[M']$, we know the same is true for $F[M']$ in $H_n(X')$ and $H_n(X', X' - x)$  for $x \in X'$. Thus, the same argument from Lemma \ref{HomEqSurj} holds here and $F$ is a surjection on $M'$ and thus also on $M$.
\end{proof}

In order to use Calabi-Weil in the finite volume case we need control over parabolic isometries which can be tricky, especially in higher dimensions where parabolic subgroups are not necessarily translation groups. We solve this problem in higher dimensions by using a stronger statement of rigidity, due to Garland and Raghunathan instead.

\begin{theorem}[Theorem 7.2 in \cite{GarlandRaghunathan70}]\label{GarlandRaghunathan}
Let $G$ be a  lattice in $Isom(\h^n)$ for $n \geq 4$. The identity representation $\rho_0: G \rightarrow Isom(\h^n):=\Lambda$ lies in a connected component of $ Hom(G; \Lambda)$ made up entirely of conjugates of $\rho_0$.
\end{theorem}

\begin{lemma} \label{CuspedPolynomials4}
Let $M$ be a triangulated finite volume hyperbolic $n$-manifold, $n \geq 4$. There is a system of polynomial inequalities as in Theorem \ref{Grigoriev} such that the following all hold:

\begin{itemize}
\item The solution set consists of cocycles in $C^1(M, SO(n,1))$

\item The solution set has a component consisting entirely of cocycles which induce faithful lattice representations.

\item The system consists of $<(n+1)^5t$ polynomials in $< (n+1)^4t$ variables, with degree bounded by $(n+1)^2t$ and all coefficients $\pm 1$.

\end{itemize}
Furthermore, this systems contains variables for the lengths of all non-ideal edges. 
\end{lemma} 

\begin{proof}

The statements and proofs of Lemmas \ref{cocyclepolys1} and \ref{cocyclepolys2} hold as is in the cusped case (for $n \geq 4$) if we replace mention of edges and simplices by non-ideal edges and non-ideal simplices.
The only other difference being that instead of applying Calabi-Weil Rigidity, we apply Theorem \ref{GarlandRaghunathan}.
\end{proof}

In the 3-dimensional case, Garland and Raghunathan's result no longer holds, and Calabi-Weil rigidity require that we have control over where parabolics are sent.  For this we consider $Isom^+(\h^3) \cong PSL(2,\C)$ instead, while sticking with the hyperboloid model, using an action described in the proof. The reason behind this is that $PSL(2,\C)$ offers a remarkably easy method of verifying parabolicity, namely, a matrix in $PSL(2,\C)$ is parabolic iff its trace is $\pm 2$, or equivalently, its trace squares to $4$.

There is one small hitch with looking at representations into $PSL(2,\C)$. To apply the results of section \ref{SecPolynomial} we need to be looking at elements of a variety. Fortunately, the following tells us we can equivalently look at representations into $SL(2,\C)$.

\begin{theorem}[Corollary 2.3 \cite{Culler86} ]
 If a discrete subgroup $G$ of $PSL(2,\C)$ has no 2-torsion then it lifts to $SL(2,\C)$.
 That is there is a homomorphism \[G \rightarrow  SL(2,\C)\] such that the composition with the natural
projection \[SL(2,\C)\rightarrow PSL(2,C)\] is the identity on $G$.
\end{theorem}

Combining this with the fact that connected components of $Hom (G, SL(2,\C))$ map down to connected components of $Hom(G, PSL(2,\C))$ tells us that if we apply the results of Section \ref{SecPolynomial} to a system such as those created so far but with images in $SL(2,\C)$, then we can find a connected component of the solution set which contains a solution inducing a faithful lattice representation, and thus by  Theorem \ref{GarlandRaghunathan} the whole component consists of such representations.

Because of the focus on parabolic elements, which correspond to loops which are freely homotopic to ideal vertices, it will be useful to be able to talk about the parabolics corresponding to each individual ideal vertex with the following definition.

\begin{defn}[Cusp Fundamental Group]
We define the fundamental group of a given cusp in $M$ to be fundamental group of the star of a chosen ideal vertex. Note that choosing a path from our chosen basepoint for $M$ to our chosen basepoint for the star of an ideal vertex defines an embedding of the cusp fundamental group into the fundamental group of $M$. All such choices define conjugate embeddings.
\end{defn}  

We can now develop a system of polynomials which suits the final case of cusped hyperbolic 3-manifolds.

\begin{lemma} \label{CuspedPolynomials3}
Let $M$ be a finite volume hyperbolic $3$-manifold triangulated by $t$ tetrahedra with fundamental group $G$. There is a system of polynomial inequalities as in Theorem \ref{Grigoriev} such that the following all hold:

\begin{itemize}
\item The solution set consists of cocycles in $C^1(M, SL(2,\C))$ which induce representations in $Hom_{par}(G, SL(2,\C))$.

\item The solution set has a component consisting entirely of cocycles each of which induces a faithful lattice representation.

\item The system has each of its measures of complexity $N, \kappa, d, M$ (as in Section  \ref{SecPolynomial}) bounded above by a constant multiple of $t$.

\end{itemize}
Furthermore, this systems contains variables for the lengths of all non-ideal edges in the image of the homotopy equivalence defined in Lemma \ref{HomotopyEquiv}. 
\end{lemma}

\begin{proof}
Outside of the handling of parabolics, the proofs here are very similar to the $SO(n,1)$ case, so we will only outline where the proofs differ from those given for Lemmas \ref{cocyclepolys1} and \ref{cocyclepolys2}.

To get a system of polynomials which has cocycles in $C^1(M, SL(2,\C))$ as its solutions (not yet including variables for edge lengths and not necessarily inducing representations in $Hom_{par}(G, SL(2,\C)$), the same proof as for $SO(n,1)$  holds. We get some constant multiple of $t$ as a bound for the number of polynomials and variables, and the polynomials have degree at most quadratic with coefficients $\pm 1$.

For edge lengths, we can again define variables for the end points of lifts of edges as well as their lengths in the hyperboloid model. The numbers of polynomials and variables added is bounded by a multiple of $t$ and the degree and coefficient complexity of the polynomials are both bounded by a constant. This uses the following action of $SL(2,\C)$ on the hyperboloid model. 

One can represent the point $(x,y,z,t)$ in the hyperboloid model by the matrix 
\[X  =  \begin{pmatrix}
t+z & x - iy \\
x + iy & t-z
\end{pmatrix}\] 
and then $A \in SL(2,\C)$ acts by

\[ X \mapsto AX A^*.\]
Note that taking determinant gives the standard quadratic form used to define the hyperboloid and the action preserves determinant.

Finally the big difference in the cusped case is that we need to check that parabolic subgroups of $\pi_1(M)$ are indeed sent to parabolics in $PSL(2,\C)$ for each representation induced by a cocycle solution to our system. In three dimensions it is enough to check that generators of the cusp fundamental group are sent to parabolics which fix the same point at infinity, as parabolics in $Isom^+(\h^3)$ restrict to translations on horospheres about their fixed point, and so products of such parabolics are also parabolics, as non-trivial products of translations are non-trivial translations. Note this is not true in higher dimensions where parabolics don't have to restrict to translations and so two parabolics with the same fixed point at infinity can have a non-trivial elliptic as their product.

To do this we take generating sets of each cusp fundamental group as described in the proof of Lemma \ref{HomotopyEquiv}. That is we pick a spanning tree of the link of each ideal vertex, then the remaining edges define loops in the link, each of which can be homotoped arbitrarily close to the ideal vertex and thus maps to a parabolic, and which together generate the fundamental group of the cusp. As there are less than $6t$ total edges, we see that there are less than $6t$ such generating curves, each of which is a product of less than $6t$ edges. 
So  for each curve we define polynomials which require that the square of the trace of the matrix associated to each such curve be $4$. As this matrix is a product, the degree of these polynomials could be as much as $6t$. We also introduce new variables which correspond to the shared fixed points of the generating set of a cusp, and polynomials requiring that they be fixed. 

At all times however, all measures of complexity of this system are bounded by a multiple of $t$.

Thus we have a system of polynomials whose solutions correspond to cocycles in $C^1(M, SL(2,\C))$ and also to the lengths of edges in the image of the homotopy equivalence. We also know that parabolic generators (and hence the entirety of each parabolic subgroup) are mapped to parabolic elements in the induced representation into $PSL(2,\C)$. We know that there is a solution to this system inducing  a  faithful lattice representation and the entire connected component of the solution set containing this solution induces representations into $PSL(2,\C)$ which all lie in the same component of $Hom_{par}(\pi_1(M), PSL(2,\C))$ and hence, by Calabi-Weil Rigidity these representations are all themselves faithful lattices
\end{proof}

As in the closed case, the following lemma follows by applying Theorem \ref{Grigoriev} to the systems of polynomials defined in Lemmas \ref{CuspedPolynomials4} and \ref{CuspedPolynomials3} and simplifying big O notation.
\begin{lemma} \label{CuspedEdgeBound}
Let $M$ be a finite volume hyperbolic manifold equipped with a semi-ideal triangulation by $t$ semi-ideal $n$-simplices. Then there exists a homotopy equivalence $F$ as described in Lemma \ref{HomotopyEquiv} such that the image of each non-ideal edge has length $l(e)$  bounded in the following way:

\[ l(e) \leq (nt)^{O(n^4t)} . \]
\end{lemma}

\begin{theorem}
Given a finite volume hyperbolic $n$-manifold $(n \geq 3)$ triangulated by $t$ semi-ideal $n$-simplices, let $R$ be the length of a systole in $M$, then
\[ R \geq \frac{1}{2^{(nt)^{O(n^4t)} }}. \]
\end{theorem}

\begin{proof}
Throughout, we let $B(t)$ be a bound on the edge length of the image of edges under the homotopy equivalence $F$.
Let $\bar{X}$ be the image in $X$ under $F$ of the non-ideal simplices in $M$. For each simplex in $M$, the diameter of the image of its non-ideal part is bounded by the maximum length of its edges, hence by $B(t)$, and thus the diameter of $\bar{X}$ is bounded by $tB(t)$. Note also that, as the links of ideal vertices consist entirely of non-ideal simplices, the images of links of ideal vertices lie in $\bar{X}$.

We aim to show that points in $X$ either lie in a bounded neighbourhood of $\bar{X}$, or they lie in the subset of the thin part made up of cusp neighbourhoods.

Note that the complement of $\bar{X}$ is covered by the images of the stars of the ideal vertices of $M$. Consider one such vertex $v_i$, and its star $st(v_i)$, let $S_i$ denote the image of $st(v_i)$ in $X$ and $L_i$ denote the image of $lk(v_i)$. Then $S_i$ is a union of geodesic semi-ideal $n$-simplices. Take a collection of lifts (one per simplex) of these semi-ideal $n$-simplices in $\h^n$ (the upper half space model) such that the collection is connected, and shares the same point on the boundary. We can perform an isometry of $\h^n$ taking this shared point to the point at infinity. We shall  call this collection $T_i$, and note that we still have a canonical map from $T_i$ into $X$. 

Consider the map $\phi_d$ which takes each point $p$ in the upper half space model and translates it a hyperbolic distance $d$ along the unique (vertical) geodesic through $p$ and the point at infinity. Then this map scales distances in $\h^n$ by 
\[ \frac{1}{e^d}. \]

Let $\psi_d$ be the map on $S_i$ which is induced by $\phi_d \mid_{T_i}$, then any edge $L_i$ has its length scaled by the scale factor above. Now as $L_i \subseteq \bar{X}$, its diameter is bounded by $tB(t)$. Thus setting 
\[ d_0 = log \frac{tB(t)}{\epsilon_n}, \]
we know that for all $d \geq d_0$, the image of $L_i$ under $\psi_d$ has diameter bounded above in the following way

\[  Diam (\psi_d(L_i )) < \frac{1}{e^{log\frac{t(B(t))}{\epsilon}}} t(B(t)) = \epsilon_n  \]

Thus as the union of all $\psi_d(L_i)$ covers $S_i$, every point $p$ lies in some $\psi_d(L_i)$. If $d\leq d_0$, then $p$ lies in a $d_0$ neighbourhood of $\bar{X}$. If $d > d_0$, then there is some curve of length $<2\epsilon_n$ based at $p$ which lies entirely in $\psi_d(L_i)$, and hence $p$ , and in fact all of $\psi_d (S_i)$ lies in the thin part of $X$. As $\psi_d(S_i)$ is a neighbourhood of an ideal vertex, it must lie in one of the cusp neighbourhoods of the thin part.

The thin part of a manifold decomposes into cusp neighbourhoods and tubes which are disjoint. Furthermore, there are no geodesics in the cusp neighbourhoods as translation towards the end of the cusp decreases the length of a curve. Thus either all systoles lie in the thick part and have length bounded below by $2 \epsilon_n$ or they lie in the thin part and are disjoint from all the cusp neighbourhoods.

Each systole must then lie in the $d_0$-neighbourhood of $\bar{X}$. This means that the entire $1$-skeleton of $\bar{X}$ lies within distance 
\[ diam(\bar{X}) + d_0 < t B(t) + log \frac{tB(t)}{\epsilon_n} \]
of each systole. By using Theorem \ref{CuspedEdgeBound} as a bound for $B(t)$, using Theorem \ref{Kellerhalls}  as a bound for $\epsilon_n$ and simplifying big O notation, we get that $\bar{X}$ lies within distance 
\[ (nt)^{O(n^4t)}. \] 
of the systole. This is therefore also an upper bound for the distance from each systole to the boundary of the tube containing it. Indeed if $\bar{X}$ lies in this tube, then a generating set for $\pi_1(X)$ lies in this subspace which has fundamental group $\Z$, and so $\pi_1(X)$ is a subgroup of $\Z$, a contradiction as $\pi_1(X)$ is the fundamental group of a finite volume hyperbolic manifold and so cannot be $\Z$.

Thus by Theorem \ref{InjRadDiameter}
\[ \frac{1}{n} log\left(\frac{1}{R}\right) +  log (\epsilon_n) -log( 4)  < (nt)^{O(n^4t)}\]

and hence we can rearrange this (simplifying big O notation) to
\[   log\left(\frac{1}{R}\right) \leq   (nt)^{O(n^4t)} + nlog(n (6\pi)^n) \]

and thus simplifying again
\[R \geq \frac{1}{2^{(nt)^{O(n^4t)}}} \]
\end{proof}

\newpage
\newcommand{\etalchar}[1]{$^{#1}$}

\end{document}